\documentclass[12pt]{amsart}    
\usepackage[colorlinks]{hyperref}
\usepackage{amsmath,amssymb,latexsym, amscd}
\usepackage{mathabx}
\usepackage{exscale, cite, eps fig, graphics}
\usepackage{mathtools}
\usepackage{graphicx}
\usepackage{subcaption}
\usepackage{enumitem}
\usepackage{lipsum}

\calclayout

\renewcommand{\geq}{\geqslant}
\renewcommand{\leq}{\leqslant}

\renewcommand{\L}{\mathcal{L}}

\renewcommand{\k}{\kappa}

\newcommand{\g}{\gamma}

\renewcommand{\b}{\beta}

\newcommand{\R}{\mathbb{R}}
\newcommand{\Z}{\mathbb{Z}}
\newcommand{\C}{\mathbb{C}}

\newtheorem{theorem}{Theorem}[section]
\newtheorem{lemma}[theorem]{Lemma}

\numberwithin{equation}{section}

\title[High-frequency instabilities]{High-frequency instabilities of the Ostrovsky equation}

\author[Bhavna]{Bhavna}
\email[Bhavna]{bhavnak@iiitd.ac.in}

\author[Kumar]{Atul~Kumar}
\email[Kumar]{atulk@iiitd.ac.in}

\author{Ashish~Kumar~Pandey}
\address{Department of Mathematics, IIIT Delhi, India 110020}
\email[Corresponding author]{ashish.pandey@iiitd.ac.in}

\date{\today}

\begin{document}

\maketitle

\begin{abstract}
We study spectral stability of small amplitude periodic traveling waves of the Ostrovsky equation. We prove that these waves exhibit spectral instabilities arising from a collision of pair of non-zero eigenvalues on the imaginary axis when subjected to square integrable perturbations on the whole real line. We also list all such collisions between pair of eigenvalues on the imaginary axis and do a Krein signature analysis.
\end{abstract}

\section{Introduction}\label{sec:intro}

The Ostrovsky equation
\begin{equation}\label{E:Ost}
(u_t - \beta u_{xxx}+(u^2)_x)_x = \g u ,\quad x\in \R\end{equation}
was derived by Ostrovsky (see \cite{Ostrovsky1978NonlinearOcean}) as a model for the unidirectional propagation of weakly nonlinear long surface and internal waves of small amplitude in rotating liquid. The liquid is assumed to be incompressible and inviscid. Here, $u(x,t)$ represents the free surface of the liquid. The constant $\gamma>0$ measures the effect of rotation and is rather small for the real conditions of the Earth rotation \cite{Galkin1991OnFluid}. The parameter $\beta$ determines the type of dispersion, namely $\beta<0$ (negative dispersion) for surface and internal waves in the ocean and surface waves in a shallow channel with an uneven bottom and $\beta>0$ (positive dispersion) for capillary waves on the surface of liquid or for oblique magneto-acoustic waves in plasma \cite{Gilman1995ApproximateEquation}. 

Setting $\g = 0$ in \eqref{E:Ost} and integrating with respect to $x \in \mathbb{R}$ and assuming the solution $u(x,t)$ and all the derivatives are vanishing at infinity, one obtains the well-known Korteweg-de Vries (KdV) equation 
\begin{equation}\label{e:kdv}
   u_t -\beta u_{xxx}+(u^2)_x = 0.
\end{equation}
The Ostrovsky equation is non-local and dispersive with linear dispersion as
\[
\omega(k) =\frac{\g}{k}+\b k^3. 
\]
It is also Hamiltonian
\[
u_t = \frac{\partial \mathcal{H}}{\partial u}
\]
where
\[
    \mathcal{H} = \int_{\mathbb{R}} \left(\frac{\beta}{2}|u_x|^2 +\frac{\g}{2} |D_x^{-1}u|^2+\frac{1}{3}u^3 \right)dx,
    \]
and for $k \in \mathbb{N}$, the operator $D_x^{-k}$ is defined by
    \[
    \widehat{(D_x^{-k}f})(\xi) = (i\xi)^{-k}\hat{f}(\xi).
    \]
The Ostrovsky equation, unlike KdV, is nonintegrable by the method of the inverse scattering transform. The local and global-well posedness of the Ostrovsky equation are known in some weighted Sobolev spaces \cite{Linares2006LocalEquation}.

The stability or instability of different type of solutions of the Ostrovsky and related models have been investigated by several authors. The orbital stability of solitary-wave solutions of the Ostrovsky equation has been established in \cite{Lu2012OrbitalEquation}. In \cite{Hakkaev2017PeriodicStability}, periodic traveling waves of \eqref{E:Ost} with general nonlinearity has been constructed for small values of $\g$ and shown to be spectrally stable to periodic perturbations of the same period as the wave.

In this article, we investigate spectral stability of small amplitude periodic traveling waves of the Ostrovsky equation. We use a standard argument based on implicit function theorem and Lyapunov-Schmidt reduction to establish the existence of a family of periodic traveling waves. As a consequence, we obtain small amplitude expansion of these periodic traveling waves. We linearize \eqref{E:Ost} about obtained periodic traveling wave and examine the $L^2(\mathbb{R})$-spectrum of the linearized operator. In case of periodic perturbations, one needs to restrict on mean zero space because of the presence of $\partial_z^{-1}$ in the linearized operator. But for square integrable perturbations on the whole real line, we use Floquet-Bloch theory which transforms $\partial_z^{-1}$ to $(\partial_z+i\xi)^{-1}$, where $\xi$ is the Floquet exponent. As a result, for $\xi\neq 0$, we need not restrict to mean zero space. In terms of perturbations, $\xi\neq 0$ corresponds to non-modulational perturbations and resulting spectral instability is termed as {\em high-frequency instability} \cite{Deconinck2017High-frequencyPDES}. We show that obtained small amplitude periodic traveling waves of \eqref{E:Ost} exhibit high-frequency instability.

In Section~\ref{sec:periodic}, we obtain periodic traveling waves of the Ostrovsky equation bifurcating from the trivial solution. We set up the spectral stability problem in Section~\ref{sec:spectral} and prove the existence of high-frequency instabilities in Section~\ref{sec:highfrequency}.

\subsection*{Notations}\label{sec:notations}
The following notations are going to be used throughout the article. Here, $L^2(\mathbb{R})$ denotes the set of real or complex valued, Lebesgue measurable functions $f(x)$ over $\mathbb{R}$ such that
\[
\|f\|_{L^2(\mathbb{R})}=\Big(\frac{1}{2\pi}\int_\R |f|^2~dx\Big)^{1/2}<+\infty \quad 
\]
and $L^2(\mathbb{T})$ denote the space of $2\pi$-periodic, measurable, real or complex valued functions over $\mathbb{R}$ such that
\[
\|f\|_{L^2(\mathbb{T})}=\Big(\frac{1}{2\pi}\int^{2\pi}_0 |f|^2~dx\Big)^{1/2}<+\infty. 
\]
For $f \in L^1(\mathbb{R})$, the Fourier transform of $f$ is written as $\hat{f}$ and defined by 
\[
\hat{f}(t)=\frac{1}{\sqrt{2\pi}}\int_{\R} f(x)e^{-itx}dx
\]
It follows from Parseval Theorem that if $f\in L^2(\mathbb{R})$ then $\|\hat{f}\|_{L^2(\mathbb{R})} = \|{f}\|_{L^2(\mathbb{R})}$.
Moreover, for any $s\in \mathbb{R}$, let $H^s(\mathbb{R})$ consist of tempered distributions such that 
\[
\|f\|_{H^s(\mathbb{R})} = \left(\int_{\R}(1+|t|^2)^s|\hat{f}(t)|^2dt\right)^{\frac{1}{2}} < +\infty
\]
Furthermore, $L^2(\mathbb{T})$-inner product is defined as
\begin{equation}\label{def:i-product}
\langle f,g\rangle=\frac{1}{2\pi}\int^{2\pi}_{0} f(z)\overline{g}(z)~dz
=\sum_{n\in\mathbb{Z}} \widehat{f}_n\overline{\widehat{g}_n}.
\end{equation}
For any $k\in \mathbb{N}$, let $H^k(\mathbb{T})$ be the space of $L^2(\mathbb{T})$ functions whose derivatives up to $k$th order are all in $L^2(\mathbb{T})$. Let $H^\infty(\mathbb{T})=\bigcap_{k=1}^\infty H^k(\mathbb{T})$.

\section{Sufficiently small and periodic traveling waves}\label{sec:periodic}
A traveling wave of \eqref{E:Ost} is a solution which propagates at a constant velocity without change of form. That is, $u(x,t)=U(x-ct)$ for some $c \in \R$. Substituting this in \eqref{E:Ost} leads to 
\begin{equation*}\label{E:U}
    cU^{\prime\prime}+\b U^{\prime\prime\prime\prime}-(U^2)^{\prime\prime}+\g U=0
\end{equation*}
We seek a $\it{periodic\hspace{3px}traveling\hspace{3px}wave}$ of \eqref{E:Ost}. That is, $U$ is a $2\pi/k$-periodic function of its argument where $k>0$ is the wave number. Taking $z:=kx$, the function $w(z):=U(kx)$ is $2\pi$-periodic in $z$ and satisfy
\begin{align}\label{E:w}
    ck^2 w''+\b k^4 w''''-k^2(w^2)''+\g w=0.
\end{align}
Note that \eqref{E:w} is invariant under $z\mapsto z+z_0$ and $z\mapsto -z$ and therefore, we may assume that $w$ is even. Also, note that \eqref{E:w} does not possess scaling invariance. Hence, we may not a priori assume that $k=1$. In fact, the stability result reported in Theorem \ref{thm:main} depends on $k$. To compare, the KdV equation \textcolor{red}{\eqref{e:kdv}} for periodic traveling waves possesses scaling invariance and stability results are independent of the carrier wave number, see \cite{Bronski2010TheEquation}, for instance. 

In what follows, we seek a non-trivial $2\pi$-periodic solution $w$ of \eqref{E:w}. For fixed $\b$ and $\g$, let $F:H^4(\mathbb{T})\times \R \times \R^+ \to L^2(\mathbb{T})$ be defined as
\begin{align}\label{E:F}
    F(w,c;k)=ck^2 w''+\b k^4 w''''-k^2(w^2)''+\g w.
\end{align}
It is well defined by a Sobolev inequality. We seek a solution $w\in H^4(\mathbb{T})$, $c \in \R$ and $k>0$ of 
\begin{equation*}\label{E:F1}
    F(w,c;k)=0.
\end{equation*}
Note that if $w \in H^4(\mathbb{T})$, then from \eqref{E:w}, $w'''' \in H^2(\mathbb{T})$ by a Sobolev inequality. Therefore, $w \in H^6(\mathbb{T})$. By a bootstrap argument, we obtain that $w \in H^\infty (\mathbb{T})$. 

The operator $F$ in \eqref{E:F} is a polynomial in parameters $c$ and $k$. Its Fr\'echet derivatives with respect to $w$ are all continuous from $H^4(\mathbb{T})$ to $L^2(\mathbb{T})$. Therefore, $F$ is a real analytic operator.

Clearly, $F(0,c;k)=0$ for all $c \in \R$ and $k>0$. If non-trivial solutions of $F(w,c;k)=0$ bifurcates from $w \equiv 0$ for some $c=c_0$ then
\begin{equation*}\label{E:df}
    L_0 := \partial_wF(0,c_0;k) = c_0k^2\partial_z^2 + \b k^4 \partial_z^4 + \g
\end{equation*} from $H^4(\mathbb{T})$ to $L^2(\mathbb{T})$, is not an isomorphism. From a straightforward calculation, 
\begin{equation*}\label{E:df1}
    L_0 e^{inz} = (-c_0k^2n^2 + \b k^4 n^4 + \g) e^{inz} = 0, \quad n \in \Z
\end{equation*}
if and only if 
\begin{equation}\label{E:c0}
    c_0=\frac{\g}{k^2 n^2}+\b k^2 n^2, \quad n\in \Z.
\end{equation}
Without loss of generality, we take $n=1$. Note that for $\beta>0$, wavenumbers, $k=\left(\frac{\g}{\b n^2}\right)^{1/4}$, $2\le n\in \mathbb{N}$, satisfy resonance condition
\[
\frac{\g}{k^2}+\b k^2 =\frac{\g}{k^2 n^2}+\b k^2 n^2
\]
of fundamental mode and $n$th harmonic then the kernel of $L_0$ is four-dimensional. For all other values of $k$, $L_0$ is a Fredholm operator of index zero with both kernal and co-kernal spanned by $e^{\pm iz}$. 

Next, we employ a Lyapunov-Schmidt procedure to establish the existence of a one-parameter family of non-trivial solutions of $F(w,c;k)=0$ bifurcating from $w \equiv 0$ and $c = c_0$. The proof follows along the same lines as the arguments in \cite{Hur2015ModulationalWaves,Hur2019ModulationalModel} and we do not include it here. We summarize the existence result for periodic traveling waves of \eqref{E:Ost} and their small amplitude expansion below.
\begin{theorem}\label{T:sol}
For any $k>0$ if $\beta<0$ and $k\neq \left(\frac{\g}{\b n^2}\right)^{1/4}$, $2\le n\in \mathbb{N}$ if $\b>0$, a one parameter family of solutions of \eqref{E:w} exists, given by $u(x,t)=w(a;k)(k(x-c(a;k)t))$ for $a \in \R$ and $|a|$ sufficiently small; $w(a;k)(\cdot)$ is $2\pi$-periodic, even and smooth in its argument, and $c(a;k)$ is even in $a$; $w(a;k)$ and $c(a;k)$ depend analytically on $a$ and $k$. Moreover, 
\begin{align*}\label{E:w_ansatz}
    w(a;k)(z)=a\cos(z) + a^2A_2\cos 2z + a^3A_3\cos 3z +a^4(A_{42}\cos 2z+A_{44}\cos 4z)+ O(a^5),
\end{align*}
and
\begin{align*}
    c(a;k)=c_0+a^2c_2+a^4c_4+O(a^6)
\end{align*}
as $a \to 0$, where $c_0$ is in \eqref{E:c0},
\[
A_2 = \dfrac{2k^2}{3\g-12\b k^4}, \quad  A_3 = \dfrac{9k^2A_2}{8\g-72\b k^4}, \quad A_{42}=2A_2A_3-2A_2^3, \quad  A_{44}=\dfrac{8k^2(A_2^2+2A_3)}{15\g-240\b k^4},
\]
\[
c_2=A_2, \quad \text{and} \quad c_4=3A_2A_3-2A_2^3.
\]
\end{theorem}

\section{Linearization and the spectral problem}\label{sec:spectral}
We linearize \eqref{E:Ost} about the solution $w$ in Theorem~\ref{T:sol} in the coordinate frame moving at the speed $c$. The result becomes
\begin{align*}
    k(v_t - ck v_z -\b k^3 v_{zzz}+2k(w v)_z)_z=\g v.
\end{align*}
We seek a solution of the form $v(z,t) = e^{\frac{\lambda}{k} t} \Tilde{v}(z)$, $\lambda\in \mathbb{C}$, to arrive at
\begin{align}\label{E:opt}
  \mathcal T^\lambda_{k,a} \Tilde{v} := (\lambda\partial_z -k^2\partial^2_z(c + \b k^2 \partial^2_z - 2w )- \g)\Tilde{v} = 0
\end{align}
The operator $\mathcal T^\lambda_{k,a}$ is defined on $L^2(\R)$ with dense domain $H^4(\mathbb{R})$. We define the spectral stability of the periodic traveling wave solution $w$ with respect to square integrable perturbations as follows: it is spectrally stable if $\mathcal T^\lambda_{k,a}$ is invertible for any $\lambda \in \mathbb{C}$ with $\Re(\lambda)>0$, otherwise, it is deemed to be spectrally unstable.

The operator $\mathcal T^\lambda_{k,a}$ has continuous spectrum in $L^2(\R)$. By Floquet theory, since coefficients of $\mathcal T^\lambda_{k,a}$ are periodic functions, all solutions of \eqref{E:opt} in $L^2(\R)$ are of the form $\Tilde{v}(z)=e^{i\xi z}V(z)$ where $\xi\in (-1/2,1/2]$ is the Floquet exponent and $V$ is a $2\pi$-periodic function, see \cite{Haragus2008STABILITYEQUATION} for a similar situation. This helps to break the invertibility problem of $\mathcal T^\lambda_{k,a}$ in $L^2(\R)$ into a family of invertibility problems in $L^2(\mathbb{T})$.

\begin{lemma}\label{lem:ft}
 The linear operator $\mathcal T^\lambda_{k,a}$ is invertible in $L^2(\R)$ if and only if linear operators
\begin{align*}\label{E:bloch}
   \mathcal T^\lambda_{k,a,\xi} = \lambda(\partial_z+i\xi) - k^2(\partial_z+i\xi)^2(c + \b k^2 (\partial_z+i\xi)^2 - 2w )- \g
\end{align*}
acting in $L^2(\mathbb{T})$ with dense domain $H^4(\mathbb{T})$ are invertible, for any $\xi\in (-1/2,1/2]$.
\end{lemma}

We refer to \cite[Proposition A.1]{Haragus2008STABILITYEQUATION} for a detailed proof in a similar situation. The $L^2(\mathbb{T})$-spectra of operators $T^\lambda_{k,a,\xi}$ consist of eigenvalues of finite multiplicity. Therefore, $T^\lambda_{k,a,\xi}$ is invertible in $L^2(\mathbb{T})$ if zero is not an eigenvalue of $T^\lambda_{k,a,\xi}$. Using this, we have the following result.

\begin{lemma}\label{lem:eq}
The operator $\mathcal T^{\lambda}_{k,a,\xi}$ is not invertible in $L^2(\mathbb{T})$ for some $\lambda\in \C$ and $\xi\neq 0$ if and only if $\lambda\in\sigma(\mathcal{A}_{k,a,\xi})$, $L^2(\mathbb{T})$-spectrum of the operator,
\begin{align*}
    \mathcal{A}_{k,a,\xi} := k^2(\partial_z+i\xi)(c + \b k^2 (\partial_z+i\xi)^2 - 2w ) + \g(\partial_z+i\xi)^{-1}.
\end{align*}
\end{lemma}
\begin{proof}
The operator $\mathcal T^{\lambda}_{k,a,\xi}$ is not invertible in $L^2(\mathbb{T})$ for some $\lambda\in \C$ and $\xi\neq 0$ if and only if zero is an eigenvalue of $\mathcal T^{\lambda}_{k,a,\xi}$. Moreover, for a $V\in L^2(\mathbb{T})$, $\mathcal T^{\lambda}_{k,a,\xi}V=0$ if and only if $\mathcal{A}_{k,a,\xi}V=\lambda V$. The proof follows trivially.
\end{proof}

Note that $\xi\neq 0$ is important in Lemma~\ref{lem:eq}. For $\xi=0$, $\mathcal{A}_{k,a,0}$ is not well-defined on $L^2(\mathbb{T})$ since $\partial_z^{-1}$ is not well-defined on $L^2(\mathbb{T})$. In what follows, we restrict $\xi$ to be non-zero and examine the $L^2(\mathbb{T})$-spectrum of $\mathcal{A}_{k,a,\xi}$. To ease the notation, we will drop $k$ from subscript in $\mathcal{A}_{k,a,\xi}$. We observe that if $\lambda \in \sigma(\mathcal A_{a,\xi})$ then $\bar{\lambda}\in \sigma(\mathcal A_{a,-\xi})$, therefore, it is enough to consider $\xi\in\left(0,1/2\right]$. Also, since $w(z)$ is even in $z$, we have
\[
\sigma (\mathcal A_{a,\xi}) = \sigma(-\mathcal A_{a,-\xi}).
\]
Consequently, we obtain spectral instability of $w$ if $\sigma (\mathcal A_{a,\xi})$ is not contained in the imaginary axis for some $\xi \in \left(0,1/2\right]$. 

A straightforward calculation shows that
\begin{align}\label{E:eigen}
    \mathcal A_{0,\xi}e^{inz}=i\omega_{n,\xi}e^{inz}, \quad n\in \Z,
\end{align}
where
\begin{align}\label{E:omega}
    \omega_{n,\xi} = k^2(n+\xi)(c_0 - \b k^2(n+\xi)^2)-\frac{\g}{n+\xi}. 
\end{align}
We have $\sigma(\mathcal A_{0,\xi})\subset i\R$ which should be the case since $a=0$ corresponds to the zero solution which is trivially stable. As $|a|$ increases, the eigenvalues in \eqref{E:eigen} move around and may leave imaginary axis to give spectral instability. Because of the symmetry of the spectrum around real and imaginary axes, spectral instability takes place only if a pair of imaginary eigenvalues collide on the imaginary axis. If the spectral instability arises from a collision away from the origin on imaginary axis, it is termed as {\em High-frequency instability} \cite{Deconinck2017High-frequencyPDES}. 

Let $n\neq m\in \Z$, and $\xi_{n,m}\in (0,1/2]$ be such that
\begin{align}\label{eq:col}
    \omega_{n,\xi_{n,m}}=\omega_{m,\xi_{n,m}}.
\end{align}
A quick calculation reveals that for a fixed value of $\g$, collisions at the origin take place only for $\b>0$, all $n\in \mathbb{Z}$, $m=-n-1$, $\xi_{n,m}=1/2$, and $k=\left(\frac{\g}{\b(n+1/2)^2}\right)^{1/4}$. There is no collision at the origin if $\beta<0$. The collision at the origin for non-zero Floquet exponent is an interesting characteristics of the Ostrovsky equation. In similar studies on other various water wave models such collisions have not been observed, see \cite{Hur2015ModulationalWaves,Hur2016ModulationalType,Hur2019ModulationalModel}, for example. 

Here, we seek to find high-frequency instabilities and therefore, lists all collisions away from the origin below.

\begin{lemma}\label{lem:collision}
The collision condition in \eqref{eq:col} is satisfied away from the origin by:
\begin{enumerate}
     \item all pairs $\{n,m\}$ except $\{-1,1\}$, and $\{-\Delta n,0\}$, $\Delta n\geq 2$ if $\beta>0$, and
    \item pairs $\{n,0\}$, $n\leq -2$,  and $\{-1,1\}$ if $\beta < 0$.
\end{enumerate}
Moreover, if one of the colliding indices $n$ and $m$ is zero then the collision occurs for wavenumbers $k\in(k_{n,m}^{\min},\infty)$ otherwise the collision occurs in a finite interval $k\in(k_{n,m}^{\min},k_{n,m}^{\max})$ for some $k_{n,m}^{\min},k_{n,m}^{\max} > 0$.
\end{lemma}
\begin{proof}
Without loss of generality, we can assume that $n<m$ and $m=n+\Delta n$ with $\Delta n\in \mathbb{N}$. Wavenumbers $k$ that satisfy \eqref{eq:col} for $\{n,n+\Delta n\}$ are given by
\[
k^4=\frac{\gamma\Delta n}{\beta}K(x,\Delta n) := \frac{\gamma\Delta n}{\beta} \frac{ 1+x(x+\Delta n) }{ x (x+\Delta n) ( (x+\Delta n)^3 - x^3 - \Delta n ) }
\]
where $x=n+\xi$. In terms of the function $K(x,\Delta n)$, since $\gamma>0$, collision between $n$ and $n+\Delta n$ takes place 
\begin{enumerate}
    \item for $\beta >0$ if $K(n+\xi, \Delta n)>0$ for some $\xi\in (0,1/2]$, and
    \item for $\beta <0$ if $K(n+\xi, \Delta n)<0$ for some $\xi\in (0,1/2]$.
\end{enumerate}
Hence, it is important to know the sign of $K(x,\Delta n)$ for a fixed $\Delta n$ and $x\in \R$. We examine this case by case.
\begin{enumerate}
    \item \underline{\textbf{Case 1 ($\Delta n=1$):}} The function $K(x,1)$, see Figure~\ref{fig:k1}, is always positive except at singularities $-1$ and $0$. Therefore, there is a collision between $n$ and $n+1$ for all $n\in \Z$ when $\beta >0$ while there is no collision between $n$ and $n+1$ for any $n\in \Z$ when $\beta <0$.
    \item \underline{\textbf{Case 2 ($\Delta n=2$):}} The function $K(x,2)$, see Figure~\ref{fig:k2}, is positive for $x\in (-\infty,-2)\cup (0,\infty)$ and negative for $x\in (-2,0)$. Therefore, there is a collision between $n$ and $n+2$ for all $n\in \Z\backslash \{-1,-2\}$ when $\beta>0$ while there is a collision between $n$ and $n+2$ only for $n=-2,$ and $-1$ when $\beta <0$.
    \item \underline{\textbf{Case 3 ($\Delta n\geq 3$):}} The function $K(x,\Delta n)$, $\Delta n\geq 3$, see Figures~\ref{fig:k3} and \ref{fig:k4} for example, is positive in
    \[
    (-\infty,-\Delta n)\cup \left(-\frac{\Delta n+\sqrt{\Delta n^2-4}}{2},-\frac{\Delta n-\sqrt{\Delta n^2-4}}{2}\right)\cup (0,\infty),
    \]
    and negative in
    \[
    \left(-\Delta n,-\frac{\Delta n+\sqrt{\Delta n^2-4}}{2}\right)\cup \left(-\frac{\Delta n-\sqrt{\Delta n^2-4}}{2},0\right).
    \]
    For $\Delta n\geq 3$, we have
    \[
    -\Delta n<-\frac{\Delta n+\sqrt{\Delta n^2-4}}{2}<-\Delta n+\frac12, \text{ and } -\frac12<-\frac{\Delta n-\sqrt{\Delta n^2-4}}{2}<0.
    \]
    Therefore, there is a collision between $n$ and $n+\Delta n$ for all  $n\in \Z\backslash\{-\Delta n\}$ when $\beta>0$ while there is a collision between $n$ and $n+\Delta n$ only for $n=-\Delta n$ when $\beta <0$.
\end{enumerate}
This proves the existence of all pairs satisfying collision condition \eqref{eq:col} away from the origin. 

Now, for a fixed $\Delta n$, if $K(n,\Delta n)>0$ for some $n\neq -\Delta n,0$ then it continue to be positive in $[n,n+1/2]$ and therefore, collision takes place between $n$ and $n+\Delta n$ for all $\xi\in (0,1/2]$. Since $K(x,\Delta n)$ restricted to $x\in [n,n+1/2]$ is a continuous function, it attains maximum and minimum in $[n,n+1/2]$ and therefore, collision takes place in a finite interval of wavenumbers $k\in(k_{n,m}^{\min},k_{n,m}^{\max})\subset (0,\infty)$, see Figure~\ref{fig:col2} for an example. For $n=-\Delta n$ or $0$, $K(x,\Delta n)$ is positive and unbounded either in $(n,n+1/2]$ and therefore $K(x,\Delta n)$ is bounded below but unbounded above. In these cases, collision takes place in an interval of wavenumbers $k\in(k_{n,m}^{\min},\infty)\subset (0,\infty)$, see Figure~\ref{fig:col1} for an example. This completes the proof.   
\end{proof}

\begin{figure}
\centering
\begin{subfigure}[b]{.45\linewidth}
\includegraphics[width=\linewidth]{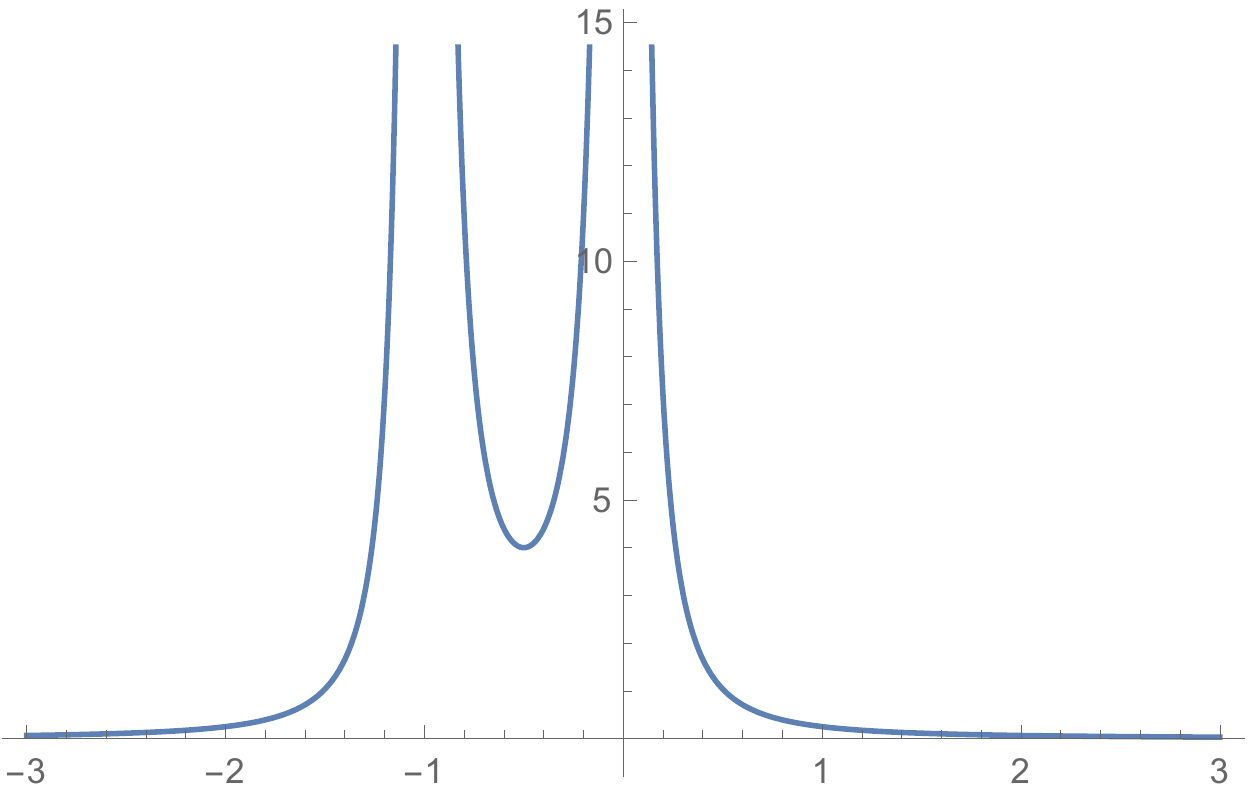}
\caption{$K(x,1)$}\label{fig:k1}
\end{subfigure}
\begin{subfigure}[b]{.45\linewidth}
\includegraphics[width=\linewidth]{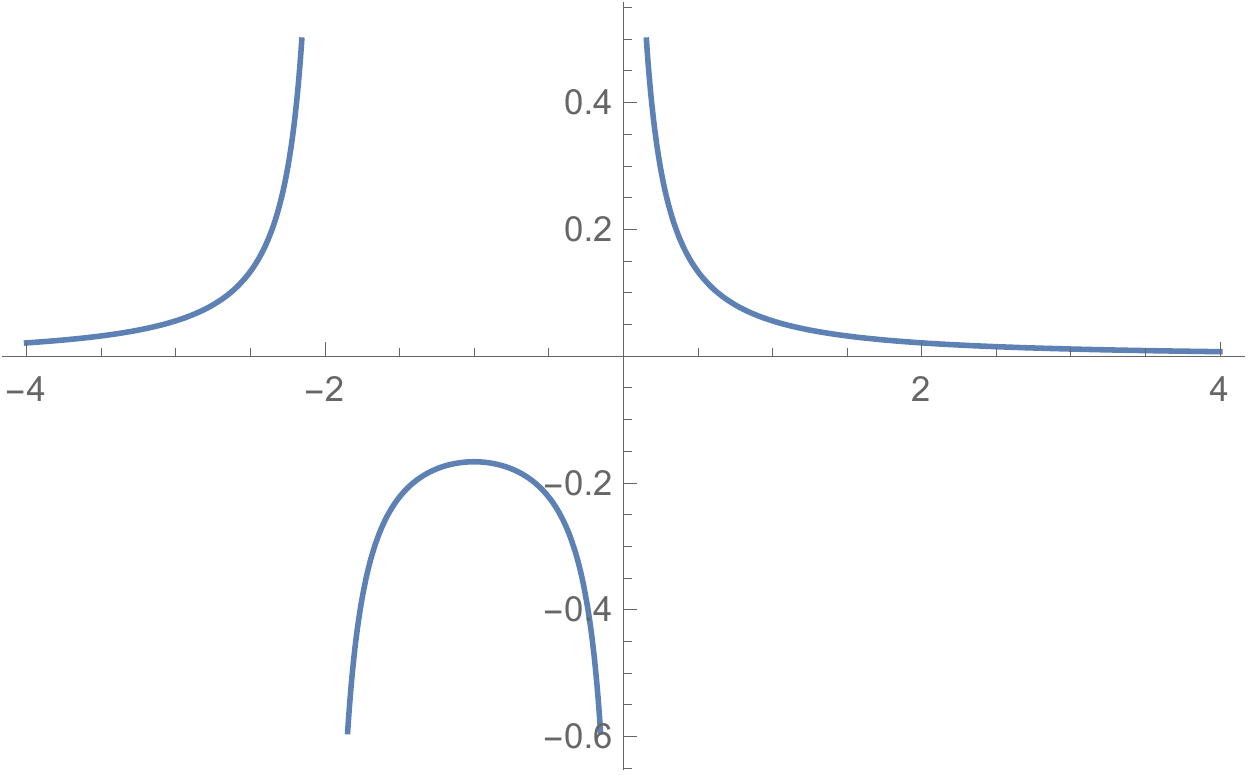}
\caption{$K(x,2)$}\label{fig:k2}
\end{subfigure}
\begin{subfigure}[b]{.45\linewidth}
\includegraphics[width=\linewidth]{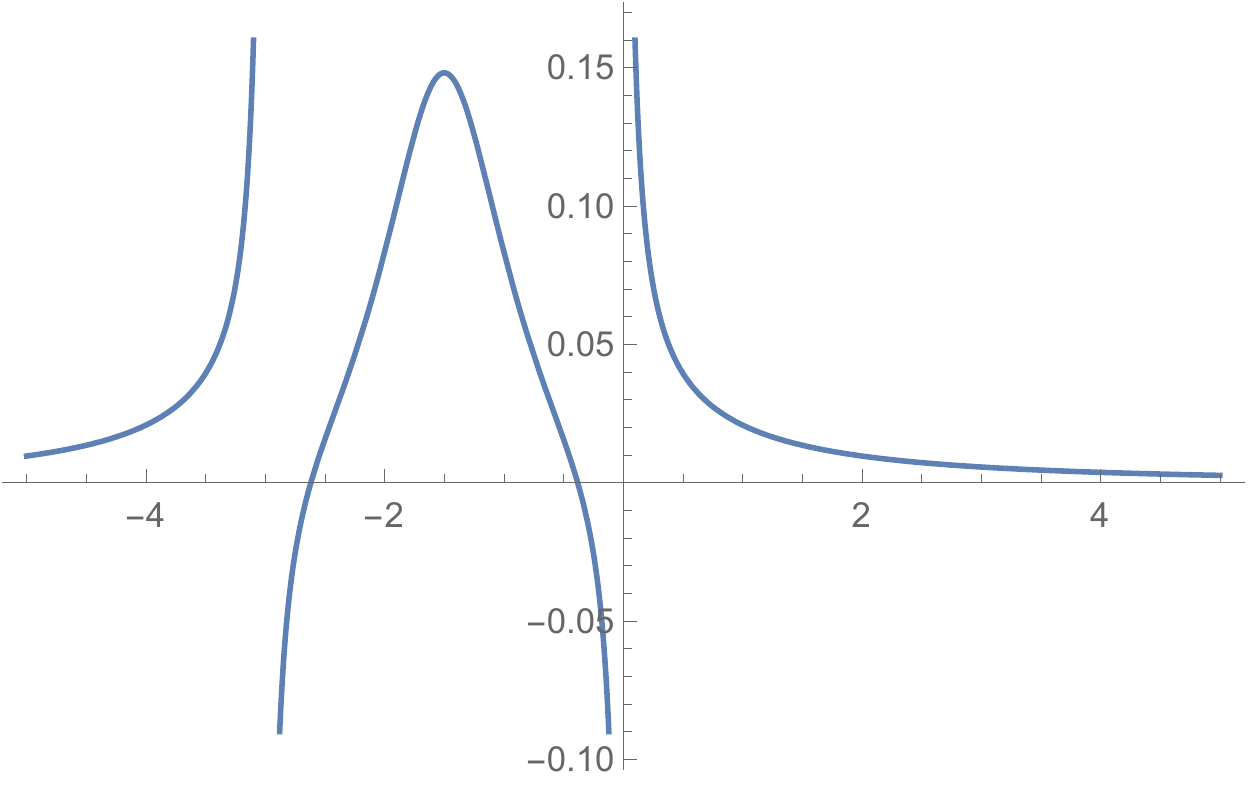}
\caption{$K(x,3)$}\label{fig:k3}
\end{subfigure}
\begin{subfigure}[b]{.45\linewidth}
\includegraphics[width=\linewidth]{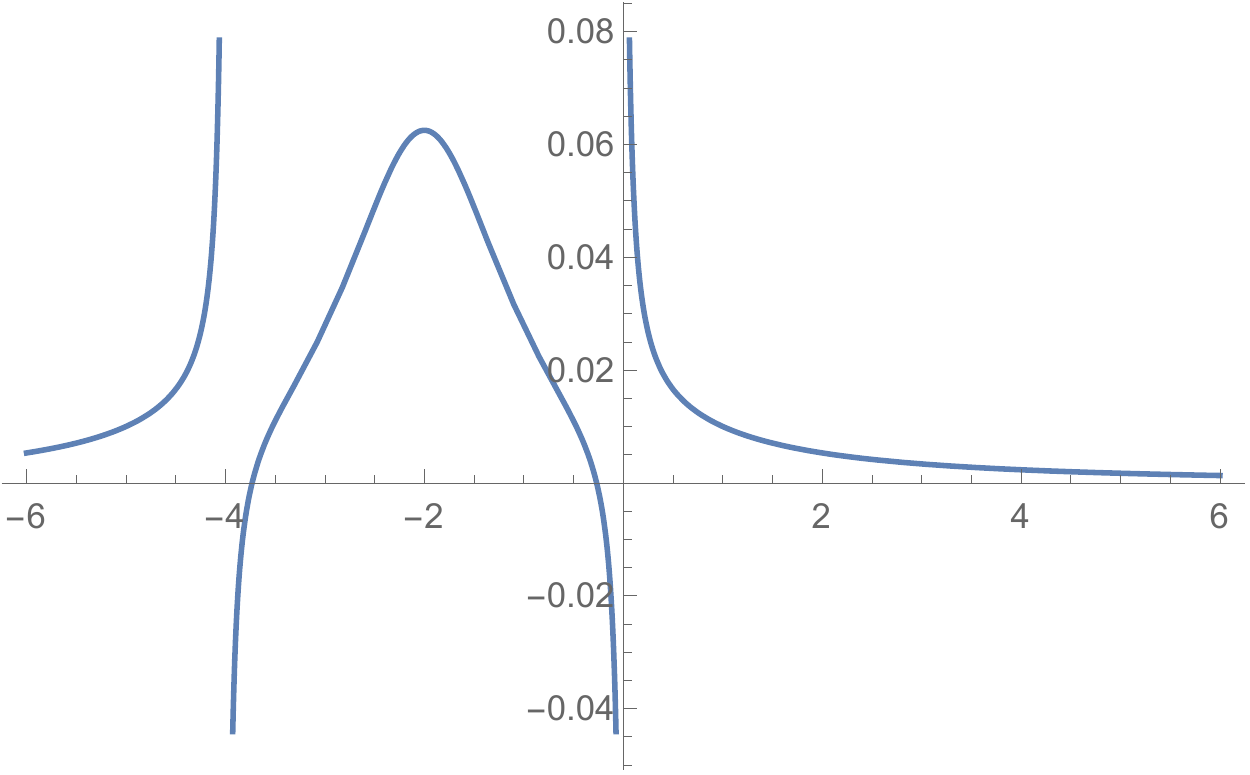}
\caption{$K(x,4)$}\label{fig:k4}
\end{subfigure}
\caption{Graph of function $K(x,\Delta n)$ vs. $x$ for $\Delta n=1,2,3$, and $4$.}
\label{fig:k}
\end{figure}

\begin{figure}
\centering
\begin{subfigure}[b]{.45\linewidth}
\includegraphics[width=\linewidth]{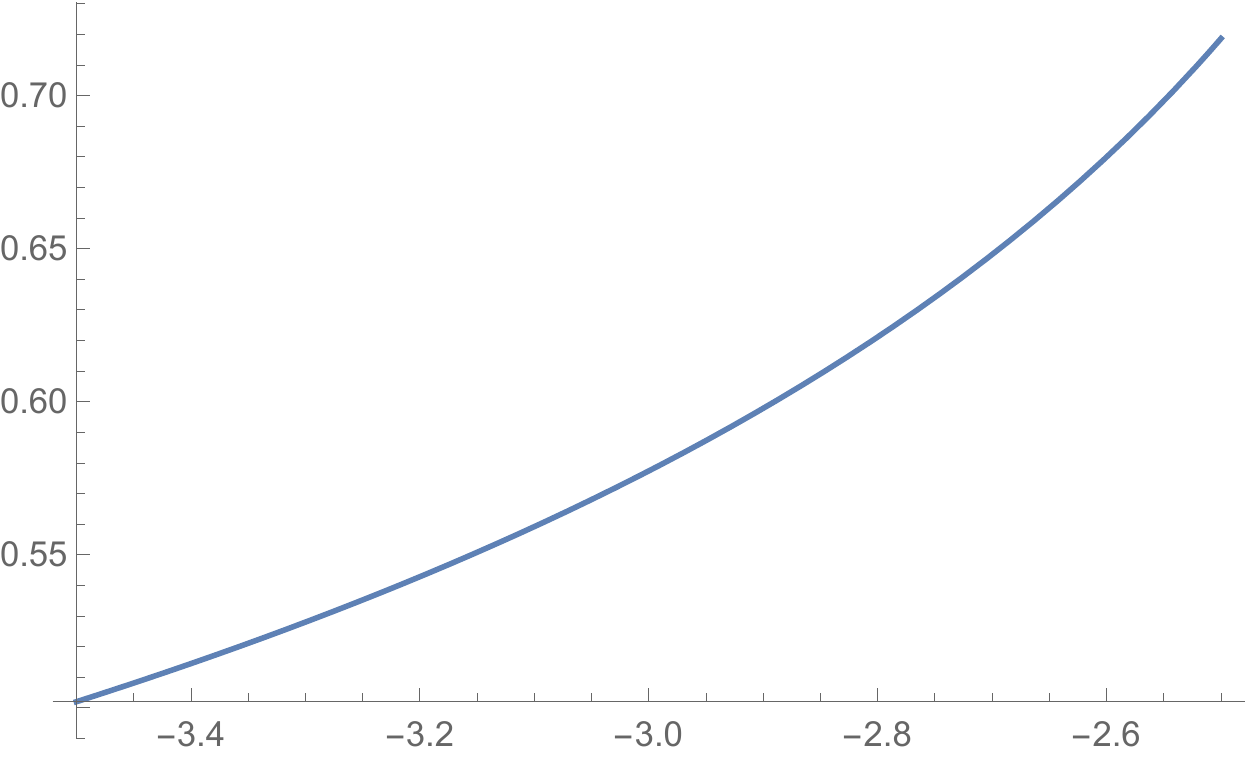}
\caption{$\gamma=1$, $\beta=1$, $n=-3$, $m=-1$}\label{fig:col2}
\end{subfigure}
\begin{subfigure}[b]{.45\linewidth}
\includegraphics[width=\linewidth]{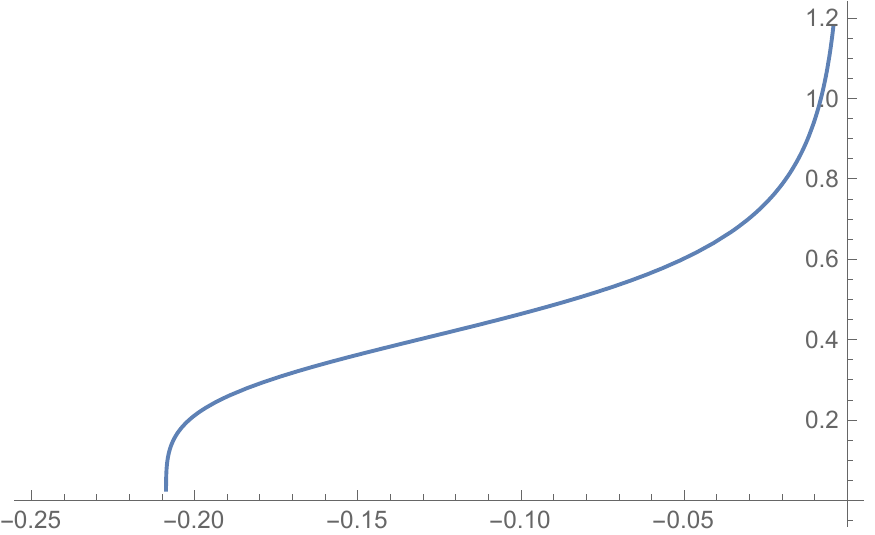}
\caption{$\gamma=1$, $\beta=-1$, $n=0$, $m=5$}\label{fig:col1}
\end{subfigure}
\caption{Graph of wavenumbers vs. $n+\xi$ for two collisions. The range of wavenumbers for which collision is taking place is approximately $(0.5,0.73)$ for the left plot and $(0, \infty)$ for the right plot.}
\label{fig:col}
\end{figure}

A necessary condition for collisions in Lemma~\ref{lem:collision} to provide high-frequency instability is that their {\em Krein signatures} at collision should be opposite \cite{MacKay1986STABILITYWAVES.}.
Since the Ostrovsky equation possesses a Hamiltonian structure, the linear operator $\mathcal{A}_{a,\xi}$ can be decomposed as 
\[
\mathcal{A}_{a,\xi} = J_\xi \L_{a,\xi}
\]
where $J_{\xi} = \partial_z + i\xi$ is skew-adjoint and
\[
 \L_{a,\xi} = k^2(c + \b k^2 (\partial_z+i\xi)^2 - 2w) + \g(\partial_z+i\xi)^{-2}
\]
is self-adjoint. With this decomposition, the Krein signature $\k_{n,\xi}$ of eigenvalues $i\omega_{n,\xi}$ in \eqref{E:omega} of $\mathcal{A}_{0,\xi}$ is given by
\begin{align}\label{eq:krein}
    \k_{n,\xi} = \operatorname{sgn}(\left<\L_{0,\xi} e^{inz}, e^{inz}\right>)=\operatorname{sgn}\left( \frac{1}{n+\xi}\omega_{n,\xi}\right)
\end{align}
where $\operatorname{sgn}$ is the signum function which determines the sign of a real number. If the collision condition \eqref{eq:col} is satisfied for some $n,m\in \Z$ and $\xi_{n,m}\in (0,1/2]$ then \eqref{eq:krein} provides that eigenvalues $i\omega_{n,\xi}$ and $i\omega_{m,\xi}$ have opposite Krein signatures at the collision if
\begin{align}\label{eq:opp}
    (n+\xi_{n,m})(m+\xi_{n,m})<0
\end{align}
otherwise they have same Krein signatures at the collision. Using \eqref{eq:opp}, we can rule out some collisions in Lemma~\ref{lem:collision} which will not lead to high-frequency instability.

\begin{lemma}\label{lem:krein}
 \begin{enumerate}
     \item For $\beta >0$, out of all collisions mentioned in Lemma~\ref{lem:collision}, $\{n,m\}$ with $n\leq -1$ and $m\geq 1$, and $\{-1,0\}$ have opposite Krein signatures.
     \item For $\beta <0$, all collisions mentioned in Lemma~\ref{lem:collision} have opposite Krein signatures.
 \end{enumerate}
\end{lemma}
\begin{proof}
If $n\neq 0$ and $m\neq 0$ then from \eqref{eq:opp}, $n$ and $m$ must be of opposite signs for \eqref{eq:opp} to hold. If one of $n$ or $m$ is zero then the other needs to be negative in order for \eqref{eq:opp} to hold. Then the proof follows.
\end{proof}

\section{High-frequency instabilities}\label{sec:highfrequency}

\begin{table}[]
    \centering
    \begin{tabular}{c|c|c}
    \hline
       $\boldsymbol{\Delta n}$  & $\boldsymbol{\beta >0}$ & $\boldsymbol{\beta< 0}$  \\\hline
         $1$ & $\{-1,0\}$ & none\\\hline
         $2$ & none & $\{-2,0\},\{-1,1\}$\\\hline
         $\geq 3$ & $\{-1,\Delta n-1\},\{-2,\Delta n-2\},\dots,\{-\Delta n+1,1\}$ & $\{-\Delta n,0\}$\\\hline
    \end{tabular}
    \caption{Collisions with opposite Krein signatures for a given $\Delta n$ for $\beta>0$ and $\beta<0$.}
    \label{tab:col}
\end{table}

Table~\ref{tab:col} summarizes all the collisions with opposite Krein signatures based on Lemma~\ref{lem:krein} for a given $\Delta n$ for both $\beta >0$ and $\beta < 0$. In what follows, we do further analysis to check if collisions in 
Table~\ref{tab:col} corresponding to $\Delta n=1$, and $2$, lead to high-frequency instability.

\subsection{$\Delta n=1$ calculation and conclusion}
For a fixed $n\in \mathbb{Z}$, let $\xi_{0}\in (0,1/2]$ be such that
\begin{align*}
     0 \neq \omega_{n,\xi_0} = \omega_{n+1,\xi_0} =: \omega. 
\end{align*}
Therefore, $i\omega$ is an eigenvalue of $\mathcal{A}_{0,\xi_0}$ of multiplicity two with an orthonormal basis of eigenfunctions $\{e^{inz},e^{i(n+1)z}\}$. For $|a|$ small, let $\lambda_{n,a,\xi_0}$ and $\lambda_{n+1,a,\xi_0}$ be eigenvalues of $\mathcal{A}_{a,\xi_0}$ bifurcating from $i\omega$ with an orthonormal basis of eigenfunctions $\{\phi_{n,a,\xi_0}(z),\phi_{n+1,a,\xi_0}(z)\}$. Note that
$
\lambda_{n,0,\xi_0}=\lambda_{n+1,0,\xi_0}=i\omega
$
with $\phi_{n,0,\xi_0}(z)=e^{inz}$ and  $\phi_{n+1,0,\xi_0}(z)=e^{i(n+1)z}$. Let
\begin{align}\label{eq:lambda}
    \lambda_{n,a,\xi_0} = i \omega + i \mu_{n,a,\xi_0}
    \quad \quad \text{and} \quad\quad
    \lambda_{n+1,a,\xi_0} = i \omega + i \mu_{n+1,a,\xi_0}.
\end{align}
We are interested in the location of $\mu_{n,a,\xi}$ and $\mu_{n+1,a,\xi}$ for $|a|$ small as if they have non-zero imaginary parts then we obtain high-frequency instability.

We start with the following expansions of eigenfunctions
\begin{align}\label{eq:eig1}
    \phi_{n,a,\xi_0} =& e^{inz}+a\phi_{n,1}+a^2\phi_{n,2}+O(a^3), \\
    \phi_{n+1,a,\xi_0} =& e^{i(n+1)z}+a\phi_{n+1,1}+a^2\phi_{n+1,2}+O(a^3).\label{eq:eig2}
\end{align}
We use orthonormality of $\phi_{n,a,\xi_0}$ and $\phi_{n+1,a,\xi_0}$ to find that 
\[
\phi_{n,1}=\phi_{n,2}=\phi_{n+1,1}=\phi_{n+1,2}=0.
\]
To trace the bifurcation of the eigenvalues from the point of the collision on the imaginary axis for $|a|$ sufficiently small, we compute the actions of $\mathcal{A}_{a,\xi_0}$ and identity operators on the extended eigenspace $\{\phi_{n,a,\xi_0}(z), \phi_{n+1,a,\xi_0}(z)\}$ viz.
\begin{align}\label{eq:bmat1}
    \mathcal{B}_{a,\xi_0} = \left[ \frac{\langle \mathcal{A}_a(\xi_0)\phi_{i,a,\xi_0}(z),\phi_{j,a,\xi_0}(z)\rangle}{\langle\phi_{i,a,\xi_0}(z),\phi_{i,a,\xi_0}(z)\rangle} \right]_{i,j=n,n+1}
\text{ and }
    \mathcal{I}_{a} = \left[ \frac{\langle \phi_{i,a,\xi_0}(z),\phi_{j,a,\xi_0}(z)\rangle}{\langle\phi_{i,a,\xi_0}(z),\phi_{i,a,\xi_0}(z)\rangle} \right]_{i,j=n,n+1}.
\end{align}
Here $\langle\hspace{2px}\cdot\hspace{2px},\hspace{2px}\cdot\hspace{2px}\rangle$ denotes the $L^2(\mathbb{T})$- inner product as defined in \eqref{def:i-product}.

Using expansions of $w$ and $c$ in Theorem~\ref{T:sol}, we expand $\mathcal{A}_{a,\xi_0}$ in $a$ as
\[
\mathcal{A}_{a,\xi_0}=\mathcal{A}_{0,\xi_0}-2ak^2(\partial_z+i\xi_0)\cos z+a^2k^2(\partial_z+i\xi_0)(c_2-2A_2\cos 2z)+O(a^3)
\]
and use the expansion of eigenfunctions in \eqref{eq:eig1}-\eqref{eq:eig2} to find the matrices in \eqref{eq:bmat1} as
\begin{align*}
    \mathcal{B}_{a,\xi_0} = 
    \begin{bmatrix}
    i \omega +ik^2a^2(n+\xi_0)c_2  &  -ik^2a (n+1+\xi_0) \\
     &  \\
    -ik^2a (n+\xi_0)  & i \omega +ik^2a^2(n+1+\xi_0)c_2
    \end{bmatrix}+O(a^3)\end{align*}
and $ 
    \mathcal{I}_{a} = 
    \begin{bmatrix}
    1   &  0 \\
    0  & 1
    \end{bmatrix}+O(a^3).$
Note that $\mathcal{B}_{0,\xi_0}=\operatorname{diag}(i\omega,i\omega)$ which should be the case as $i\omega$ is an eigenvalue of $\mathcal{A}_{0,\xi_0}$ of multiplicity two.
The two values of $\mu$ solving the equation 
\begin{align}\label{eq:cheq1p}
    \det(\mathcal{B}_{a,\xi_0}-(i \omega + i \mu) \mathcal{I}_{a}) = 0,
\end{align}
would coincide with $\mu_{n,a,\xi_0}$ and $\mu_{n+1,a,\xi_0}$ in \eqref{eq:lambda} in leading order of $a$. 
Plugging the values in \eqref{eq:cheq1p} and calculating the discriminant of the quadratic in $\mu$, we arrive at
\begin{align*}
    \mathbb{D}_{a,\xi_0} = 4k^4a^2(n+\xi_0)(n+1+\xi_0)+O(a^3). 
\end{align*}
Therefore, for sufficiently small $|a|$,  if $(n+\xi_0)(n+1+\xi_0)$ is negative then we would obtain high-frequency instability.

From Table~\ref{tab:col}, the only collision for $\Delta n=1$ is when $\beta>0$ and $n=-1$. This collision takes place for all values of $\xi\in(0,1/2]$, see Figure~\ref{fig:cc_10A}.
\begin{figure}
\centering
\includegraphics[width=8cm,height=6cm]{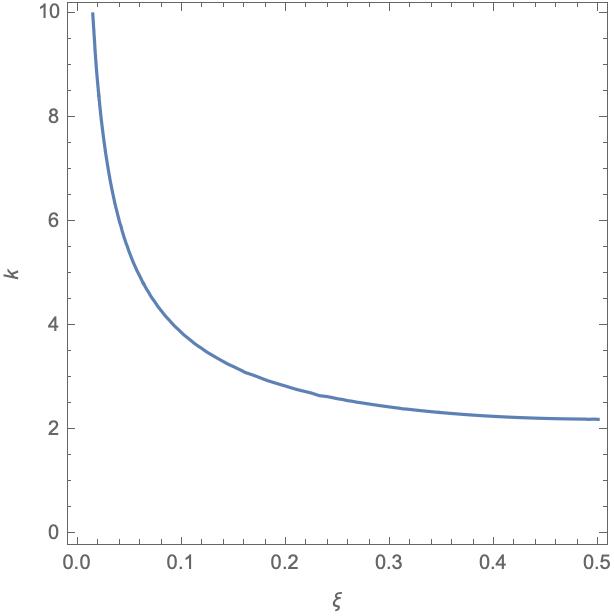}
\caption{Collision contour describing collision between eigenvalues $i\omega_{-1,\xi}$ and $i\omega_{0,\xi}$ for different values of $k$ and $\xi$ for $\gamma=6$ and $\beta=1$.}\label{fig:cc_10A}
\end{figure}
Analyzing the function $K(x,\Delta n)$ for $x=\xi-1$ and $\Delta n=1$, we easily deduce that this collision takes place for wavenumbers $k\in ((4\gamma/\beta)^{1/4},\infty)$. We summarize the result in the following theorem.

\begin{theorem}\label{thm:main}
For a fixed $\gamma>0$ and $\beta>0$, a $2\pi/k$-periodic traveling wave of \eqref{E:Ost} given by $u(x,t)=w(k(x-ct))$ where $w$ and $c$ are given in Theorem~\ref{T:sol} suffers high-frequency instability if 
\[
k>\sqrt[4]{\dfrac{4\gamma}{\beta}}.
\]
\end{theorem}

\subsection{$\Delta n=2$ calculation and conclusion}
We proceed as in the previous section. For a fixed $n\in \mathbb{Z}$, let $\xi_{0}\in (0,1/2]$ be such that
\begin{align*}
     0 \neq \omega_{n,\xi_0} = \omega_{n+2,\xi_0} =: \omega. 
\end{align*}
That is, $i\omega$ is an eigenvalue of $\mathcal{A}_{0,\xi_0}$ of multiplicity two with an orthonormal basis of eigenfunctions $\{e^{inz},e^{i(n+2)z}\}$. As before, for $|a|$ small, let $\lambda_{n,a,\xi_0}$ and $\lambda_{n+2,a,\xi_0}$ be eigenvalues of $\mathcal{A}_{a,\xi_0}$ bifurcating from $i\omega$ with an orthonormal basis of eigenfunctions $\{\phi_{n,a,\xi_0}(z),\phi_{n+2,a,\xi_0}(z)\}$. Let
\begin{align}\label{eq:lambda2}
    \lambda_{n,a,\xi_0} = i \omega + i \mu_{n,a,\xi_0}
    \quad \quad \text{and} \quad\quad
    \lambda_{n+2,a,\xi_0} = i \omega + i \mu_{n+2,a,\xi_0}
\end{align}
and we are interested in the location of $\mu_{n,a,\xi}$ and $\mu_{n+2,a,\xi}$ for $|a|$ small. Again, using orthonormality of $\phi_{n,a,\xi_0}$ and $\phi_{n+2,a,\xi_0}$ we find that 
\begin{align}\label{eq:eig11}
    \phi_{n,a,\xi_0} = e^{inz}+O(a^5)\quad \text{ and }\quad
    \phi_{n+2,a,\xi_0} = e^{i(n+2)z}+O(a^5).
\end{align}
As before, we compute the action matrices of $\mathcal{A}_{a,\xi_0}$ and identity operators on the extended eigenspace $\{\phi_{n,a,\xi_0}(z), \phi_{n+2,a,\xi_0}(z)\}$. We use expansions of $w$ and $c$ in Theorem~\ref{T:sol} to expand $\mathcal{A}_{a,\xi_0}$ in $a$ as
\begin{align*}
\mathcal{A}_{a,\xi_0}=&\mathcal{A}_{0,\xi_0}-2ak^2(\partial_z+i\xi_0)\cos z+a^2k^2(\partial_z+i\xi_0)(c_2-2A_2\cos 2z)-2a^3k^2A_3(\partial_z+i\xi_0)\cos 3z\\
&+a^4k^2(\partial_z+i\xi_0)(c_4-2(A_{42}\cos 2z+A_{44}\cos 4z))+O(a^5).
\end{align*}
Using the expansion of eigenfunctions in \eqref{eq:eig11}, the matrices in \eqref{eq:bmat1} turn out to be
\begin{align*}
    \mathcal{B}_{a,\xi_0} = 
    \begin{bmatrix}
    i \omega + ik^2(a^2A_2+a^4c_4)(n+\xi_0)  &  -ik^2(a^2A_2+a^4A_{42}) (n+2+\xi_0) \\
     &  \\
     -ik^2(a^2A_2+a^4A_{42})(n+\xi_0)  & i \omega + ik^2(a^2A_2+a^4c_4)(n+2+\xi_0)
    \end{bmatrix}+O(a^5)
\end{align*}
and
$
    \mathcal{I}_{a} = 
    \begin{bmatrix}
    1   &  0 \\
    0  & 1
    \end{bmatrix}+O(a^5).
$
Again, we solve the equation 
\begin{align*}
    \det(\mathcal{B}_{a,\xi_0}-(i \omega + i \mu) \mathcal{I}_{a}) = 0,
\end{align*} 
to obtain a quadratic in $\mu$
whose discriminant is given by
\begin{align*}
    \mathbb{D}_{a,\xi_0} = 4k^4a^4A_2^2(n+\xi_0+1)^2 + O(a^5).
\end{align*}
Note that, irrespective of the values of $n$ and $\xi_0$, the leading term in the discriminant is always positive. Therefore, we do not observe any high-frequency instability for $\Delta n=2$ case by performing the perturbation calculation up to fourth power of the amplitude parameter $a$.
\section*{Acknowledgement} Bhavna and AKP are supported by the Science and Engineering Research Board (SERB), Department of Science and Technology (DST), Government of India under grant SRG/2019/000741. AK is supported by Junior Research Fellowships (JRF) by Council of Scientific and Industrial Research (CSIR), Government of India.

\section*{Conflict of interest statement}
On behalf of all authors, the corresponding author states that there is no conflict of interest.

\section*{Data availability statement}
Data sharing not applicable to this article as no datasets were generated or analysed during the current study.

\bibliographystyle{amsalpha}
\bibliography{Ost.bib}

\providecommand{\bysame}{\leavevmode\hbox to3em{\hrulefill}\thinspace}
\providecommand{\MR}{\relax\ifhmode\unskip\space\fi MR }
\providecommand{\MRhref}[2]{%
  \href{http://www.ams.org/mathscinet-getitem?mr=#1}{#2}
}
\providecommand{\href}[2]{#2}
\begin{thebibliography}{LLW12}

\bibitem[BJ10]{Bronski2010TheEquation}
Jared~C. Bronski and Mathew~A. Johnson, \emph{{The modulational instability for
  a generalized korteweg-de vries equation}}, Archive for Rational Mechanics
  and Analysis \textbf{197} (2010), no.~2.

\bibitem[DT17]{Deconinck2017High-frequencyPDES}
Bernard Deconinck and Olga Trichtchenko, \emph{{High-frequency instabilities of
  small-amplitude solutions of Hamiltonian PDES}}, Discrete and Continuous
  Dynamical Systems- Series A \textbf{37} (2017), no.~3.

\bibitem[GGS95]{Gilman1995ApproximateEquation}
O.~A. Gilman, R.~Grimshaw, and Yu.~A. Stepanyants, \emph{{Approximate
  Analytical and Numerical Solutions of the Stationary Ostrovsky Equation}},
  Studies in Applied Mathematics \textbf{95} (1995), no.~1.

\bibitem[GS91]{Galkin1991OnFluid}
V.~N. Galkin and Yu~A. Stepanyants, \emph{{On the existence of stationary
  solitary waves in a rotating fluid}}, Journal of Applied Mathematics and
  Mechanics \textbf{55} (1991), no.~6.

\bibitem[Har08]{Haragus2008STABILITYEQUATION}
Mariana Haragus, \emph{{Stability of periodic waves for the generalized BBM
  equation}}, Rev. Roumaine Maths. Pures Appl. \textbf{53} (2008), 445--463.

\bibitem[HJ15]{Hur2015ModulationalWaves}
Vera~Mikyoung Hur and Mathew~A. Johnson, \emph{{Modulational Instability in the
  Whitham Equation for Water Waves}}, Studies in Applied Mathematics
  \textbf{134} (2015), no.~1.

\bibitem[HP16]{Hur2016ModulationalType}
Vera~Mikyoung Hur and Ashish~Kumar Pandey, \emph{{Modulational instability in
  nonlinear nonlocal equations of regularized long wave type}}, Physica D:
  Nonlinear Phenomena \textbf{325} (2016).

\bibitem[HP19]{Hur2019ModulationalModel}
\bysame, \emph{{Modulational instability in a full-dispersion shallow water
  model}}, Studies in Applied Mathematics \textbf{142} (2019), no.~1, 3--47.

\bibitem[HSS17]{Hakkaev2017PeriodicStability}
Sevdzhan Hakkaev, Milena Stanislavova, and Atanas Stefanov, \emph{{Periodic
  traveling waves of the regularized short pulse and Ostrovsky equations:
  Existence and stability}}, SIAM Journal on Mathematical Analysis \textbf{49}
  (2017), no.~1.

\bibitem[LLW12]{Lu2012OrbitalEquation}
Dianchen Lu, Lili Liu, and Li~Wu, \emph{{Orbital stability of solitary waves
  for generalized Ostrovsky equation}}, Lecture Notes in Electrical
  Engineering, vol. 136 LNEE, 2012.

\bibitem[LM06]{Linares2006LocalEquation}
Felipe Linares and Aniura Milan{\'{e}}s, \emph{{Local and global well-posedness
  for the Ostrovsky equation}}, Journal of Differential Equations \textbf{222}
  (2006), no.~2.

\bibitem[MS86]{MacKay1986STABILITYWAVES.}
R.~S. MacKay and P.~G. Saffman, \emph{{STABILITY OF WATER WAVES.}}, Proceedings
  of The Royal Society of London, Series A: Mathematical and Physical Sciences
  \textbf{406} (1986), no.~1830.

\bibitem[Ost78]{Ostrovsky1978NonlinearOcean}
L.~A. Ostrovsky, \emph{{Nonlinear internal waves in a rotating ocean}},
  Oceanology \textbf{18} (1978), no.~2.

\end{thebibliography}
\end{document}